\documentclass[10pt,reqno]{amsart}
\usepackage{amsmath}
\usepackage{amssymb}
\usepackage[left=3cm,top=3cm,right=3cm,bottom=3cm]{geometry}
\usepackage{epsfig}
\usepackage{todonotes}

\begin{document}
\newtheorem{theorem}{Theorem}[section]
\newtheorem{lemma}[theorem]{Lemma}
\newtheorem{definition}[theorem]{Definition}
\newtheorem{conjecture}[theorem]{Conjecture}
\newtheorem{proposition}[theorem]{Proposition}
\newtheorem{algorithm}[theorem]{Algorithm}
\newtheorem{corollary}[theorem]{Corollary}
\newtheorem{observation}[theorem]{Observation}
\newtheorem{problem}[theorem]{Open Problem}
\newtheorem{remark}[theorem]{Remark}
\newcommand{\noin}{\noindent}
\newcommand{\ind}{\indent}
\newcommand{\om}{\omega}
\newcommand{\I}{\mathcal I}
\newcommand{\pp}{\mathcal P}
\newcommand{\ppp}{\mathfrak P}
\newcommand{\N}{{\mathbb N}}
\newcommand{\LL}{\mathbb{L}}
\newcommand{\R}{{\mathbb R}}
\newcommand{\E}[1]{\mathbb{E}\left[#1 \right]}
\newcommand{\V}{\mathbb Var}
\newcommand{\Prob}{\mathbb{P}}
\newcommand{\eps}{\varepsilon}

\newcommand{\Tv}{P}

\newcommand{\mT}{\mathcal{T}}
\newcommand{\mS}{\mathcal{S}}
\newcommand{\mA}{\mathcal{A}}
\newcommand{\mB}{\mathcal{B}}

\newcommand{\Cyc}[1]{\mathrm{Cyc}\left(#1\right)}
\newcommand{\Seq}[1]{\mathrm{Seq}\left(#1\right)}
\newcommand{\Mul}[1]{\mathrm{Mul}\left(#1\right)}
\newcommand{\Mulo}[1]{\mathrm{Mul}_{>0}\left(#1\right)}
\newcommand{\Set}[1]{\mathrm{Set}\left(#1\right)}
\newcommand{\Setd}[1]{\mathrm{Set}_{d}\left(#1\right)}

\title{On the limiting distribution of the metric dimension for random forests}

\author{Dieter Mitsche}
\address{(DM) Laboratoire J.A. Dieudonn\'{e}, Universit\'{e} de Nice Sophia-Antipolis. Parc Valrose, 06108 Nice cedex 02, France}
\email{\tt dmitsche@unice.fr}

\author{Juanjo Ru\'e}
\address{(JR) Freie Universit\"at Berlin, Institut f\"ur Mathematik und Informatik, Arnimallee 3, 14195 Berlin, Germany}
\email{jrue@zedat.fu-berlin.de}

\thanks{J.\,R.~was partially supported by the Spanish MICINN grant MTM2011-22851, the FP7-PEOPLE-2013-CIG project \emph{CountGraph} (ref. 630749), the DFG within the research training group \emph{Methods for Discrete Structures} (GRK1408), and the \emph{Berlin Mathematical School.}}


\begin{abstract}

The metric dimension of a graph $G$ is the minimum size of a subset $S$ of vertices of $G$ such that all other vertices are uniquely determined by their distances to the vertices in $S$. In this paper we investigate the metric dimension for two  different models of random forests, in each case obtaining normal limit distributions for this parameter.
\end{abstract}

\maketitle

\section{Introduction}\label{sec:intro}

Let $G=(V,E)$ be a finite, simple graph with $|V|=n$ vertices.
For a subset $R \subseteq V$ with $|R|=r$, and a vertex $v \in V$, let $d_R(v)$ be the $r$-dimensional vector whose $i$-th coordinate is the length of the shortest path between $v$ and the $i$-th vertex of $R$. If no such path exists because the considered vertices are in different connected components, the distance is defined to be $\infty$. We call $R \subseteq V$ a \emph{resolving set} if for any pair of vertices $u, v \in V$, $d_R(u) \neq d_R(v)$. For instance, the full vertex set $V$ is always a resolving set, and so is $R=V\setminus \{w\}$ for every choice of $w$. The general problem in this domain is to find \emph{minimal} resolving sets. The \emph{metric dimension} $\beta(G)$ of a connected graph $G$ with $n \geq 2$ vertices (or simply $\beta$, if the graph we consider is clear from the context) is then the smallest cardinality of a resolving set. If $G$ is an isolated vertex, then we define $\beta(G):=1$. Observe also that for a graph $G$ with connected components $G_1, \ldots, G_k$, $k \geq 2$, none of them being an isolated vertex, we have $\beta(G)=\sum_{i=1}^k \beta(G_i)$: in order to distinguish two vertices from the same connected component, a minimal resolving set of this connected component has to be chosen. If on the other hand $G$ has connected components $G_1,\ldots,G_k$ and at least one isolated vertex, then $\beta(G)=\left(\sum_{i=1}^k \beta(G_i)\right) -1$, as one isolated vertex is distinguished from all others without choosing the vertex: it will be the only vertex at distance $\infty$ from everyone else. Note that for a graph $G$ on $n \geq 2$ vertices we have the trivial inequalities $1 \leq \beta(G) \leq n-1$, with the lower bound attained for a path of length $n$, and the upper bound for the complete graph $K_n$ (or the empty graph).

This parameter was initially introduced by Slater~\cite{Sla75}, and Harary and Melter~\cite{Har76}. As a start, Slater~\cite{Sla75} determined a characterization of the metric dimension of trees, which was then independently rediscovered by Harary and Melter~\cite{Har76}: for any tree $T$ on $n$ vertices which is not a path, the metric dimension of $T$ is $|L|-|K|$, where $L$ is the set of leaves of $T$ and $K$ is the set of vertices that have degree greater than two and that are connected by paths whose interior vertices are degree-two-vertices to one or more leaves. Moreover, this characterization is constructive: one can find a resolving set of size $|L|-|K|$ by removing from $L$ one of the leaves associated with each vertex in $K$.

The same result is obtained by means of the characterization of the metric dimension for trees of Kuller et al.~\cite{Kul96}: for any tree $T$ which is not a path, and a vertex $v \in V(T)$, and any two edges $e, f \in E(T)$, the equivalence relation $ =_v$ is defined as follows: $e =_v f$ iff there is a path in $T$ including $e$ and $f$ that does not have $v$ as an internal vertex. The subgraphs induced by the edges of the equivalence classes of $E(T)$ are called the \emph{bridges} of $T$ relative to $v$. Define then the \emph{legs} at $v$ to be the bridges which are paths, and denote by $\ell_v$ the number of legs at $v$. The metric dimension of the tree then satisfies the relation $\beta(T)=\sum_{v: \ell_v > 1} (\ell_v - 1)$.

Two decades later, Khuller,  Raghavachari and Rosenfeld~\cite{Kul96} gave a linear-time algorithm for computing the metric dimension of a tree, and they characterized the graphs with metric dimensions $1$ and $2$ (in the first case, paths are the unique graphs with metric dimension equal to 1). Later on, on the other end, Chartrand,  Eroh, Johnson and Oellermann~\cite{CEJO00} gave necessary and sufficient conditions for a graph $G$ to satisfy $\beta(G)=n-1$ or $\beta(G)=n-2$.

The metric dimension has deep connections with other graph parameters: denoting by $D=D(G)$ the diameter of a graph $G$, it was observed in~\cite{Kul96} that $n \leq D^{\beta-1}+\beta$. Recently, Hernando,  Mora,  Pelayo,  Seara and Wood~\cite{HMPSW10} proved that $n \leq (\lfloor \frac{2D}{3}\rfloor +1)^{\beta}+\beta \sum_{i=1}^{\lceil D/3 \rceil} (2i-1)^{\beta-1}$, and  gave extremal constructions that show that this bound is sharp. Moreover, in~\cite{HMPSW10} graphs of metric dimension $\beta$ and diameter $D$ were characterized. The metric dimension of the cartesian product of graphs was investigated by C\'{a}ceres, Hernando et al.~\cite{Cac07}, and the relationship between $\beta(G)$ and the {\em determining number} of $G$ (the smallest size of a set $S$ such that every automorphism of $G$ is uniquely determined by its action on $S$) was studied by C\'{a}ceres, Garijo et al.~\cite{CGPS10}. Also, Bailey and Cameron~\cite{Bailey11} studied the  metric dimension of graphs with strong symmetry properties, such as distance transitive graphs (where the orbits on pairs of vertices are precisely the distance classes).

Concerning algorithmic questions, the problem of finding the metric dimension is known to be NP-complete for general graphs (see~\cite{GJ79, Kul96}). Recently, D\'{\i}az et al.~\cite{Diaz12} showed that determining $\beta(G)$ is NP-complete for planar graphs, and the authors also gave a polynomial-time algorithm for determining the metric dimension of an outerplanar graph. Furthermore, in~\cite{Kul96} a polynomial-time algorithm approximating $\beta(G)$ within a factor  $2 \log n$ was given. On the other hand, Beerliova et al.~\cite{Beerliova} showed that the problem is inapproximable within $o(\log n)$ unless P=NP. Hauptmann et al.~\cite{Hauptmann} then strengthened the result and showed that unless NP $\subseteq$ DTIME$(n^{\log \log n})$, for any $\eps > 0$, there is no $(1-\eps) \log n$-approximation for determining $\beta(G)$, and finally Hartung et al.~\cite{Hartung12} extended the result by proving that the metric dimension problem is still inapproximable within a factor of $o(\log n)$ on graphs with maximum degree three.

In this paper we study the metric dimension of forests in different random models. Our first contribution is the analysis of the limiting probability of the metric dimension for a random tree, chosen uniformly at random among all trees with $n$ vertices. The same result applies for random planar forests in the corresponding similar model. These models are reminiscent of the random planar graph model introduced by Denise, Vanconcellos and Welsh~\cite{DVW} (see also~\cite{Rpg}). All asymptotic results throughout the following lines are as $n \rightarrow \infty $. In particular, our first result is the following one:

\begin{theorem}\label{thm:main-tree}
Let $T_n, F_n$ be a random tree (respectively random forest) chosen uniformly at random among all trees (respectively forests) with $n$ vertices. Then, each of the sequences of random variables
$$\frac{\beta(T_n) - \E{\beta(T_n)}}{\sqrt{\V{\beta(T_n)}}},\,\, \frac{\beta(F_n) - \E{\beta(F_n)}}{\sqrt{\V{\beta(F_n)}}} $$
converge in distribution to a standard normal distribution when $n\rightarrow \infty$. Additionally, $\E{\beta(T_n)}= \E{\beta(F_n)} = \mu n (1+o(1)) $ and $\V{\beta(T_n)} = \V{\beta(F_n)}  =\sigma^2 n (1+o(1))$, and $\mu \simeq 0.14076941$, $\sigma^2 \simeq 0.063748151$.
\end{theorem}
We also study random forests in the context of the Erd\H{o}s-R\'enyi model $G(n,p)$ for random graphs. Many results are also known in this context: Babai et al.~\cite{Babai80} showed that in $G(n,1/2)$ asymptotically almost surely the set of $\lceil (3\log n)/\log 2 \rceil$ vertices with the highest degrees can be used to test whether two random graphs are isomorphic (in fact they gave an algorithm with running time in $O(n^2)$), and hence they obtained an upper bound of $\lceil (3\log n)/\log 2 \rceil$ for the metric dimension of $G(n,1/2)$. Next, Frieze et al.~\cite{Frieze07} studied sets resembling resolving sets, namely \emph{identifying codes}: a set $C \subseteq V$ is an identifying code of $G$, if $C$ is a dominating set (every vertex $v \in V \setminus C$ has at least one neighbor in $C$) and $C$ is also a separating set (for all pairs $u,v \in V$, one must have $N[u] \cap C \neq N[v] \cap C$, where $N[u]$ denotes the closed neighborhood of $u$). Observe that a graph might not have an identifying code, but note also that for random graphs with diameter $2$ the concepts are very similar. The existence of identifying codes and bounds on their sizes in $G(n,p)$ were established in~\cite{Frieze07}. The same problem in the model of random geometric graphs was analyzed by M\"uller and Sereni~\cite{Mueller09}, and Foucaud and Perarnau~\cite{Foucaud12} studied the same problem in random $d$-regular graphs. Finally, in a recent paper~\cite{BMP13} the authors studied the metric dimension of $G(n,p)$ for a wide range of values of $$(\log n)^5 \ll p(n-1) \leq n\left(1 - \frac{3 \log \log n}{\log n} \right).$$ In this last work the authors showed a zigzag-behavior of $\beta(G)$ in terms of the edge probability $p$.

The second contribution of this paper is the analysis of the metric dimension of sparse $G(n,p)$ with $p = \frac{c}{n}$ with $c < 1$. This range of parameters typically has a very forest-like structure, although a few cycles might be present. In such a situation the behavior is quite regular, and indeed we can obtain precise limiting distributions for this parameter. To make our result precise, we need the following notation. Let $F_n$ a distribution function of a certain random variable and let $\Phi$ denote the distribution function of the standard normal law. Define the following measure of convergence
$$
d(F_n,\Phi)=\sup_{h(x)} \frac{\left| \int h(x)dF_n(x) - \int h(x) d\Phi(x) \right|}{|| h ||},
$$
where $||h||=\sup_x |h(x)|+\sup_x |h'(x)|,$   and the supremum is taken over all bounded test functions $h$ with bounded derivative. For a random variable $X$ denote by $\mathcal{L}(X)$ its distribution function (if it exists). We also say that a property holds \emph{asymptotically almost surely}, or \emph{a.a.s.}, if the probability for it to hold tends to $1$ as $n \to \infty$.


\begin{theorem}\label{thm:main-Gnp} Let $G \in G(n,p)$.
\item[(i)] For $p=o\left(n^{-1}\right)$, $\beta(G)=n(1+o(1))$ a.a.s.
\item[(ii)] For $p=\frac{c}{n}$ with $0 < c < 1$, the sequence of random variables $$\frac{\beta(G) - \E{\beta(G)}}{\sqrt{\V{\beta(G)}}}$$ converges in distribution to a standard normal distribution as $n\rightarrow \infty$, and
$$ d\left( \mathcal{L}\left( \frac{\beta(G) - \E{\beta(G)}}{\sqrt{\V{\beta(G)}}}\right), \Phi \right) = O\left(n^{-1/2}\right).$$
Moreover, $\E{\beta(G)}=Cn(1+o(1))$, where
\begin{equation}\label{eq:C}
C=e^{-c}\left( \frac32 +c+\frac{c^2}{2} - e^c-\frac12 e^{ce^{-c}}+exp\left(c \frac{1-(c+1)e^{-c}}{1-ce^{-c}}\right) - c \frac{e^{-c}}{1-ce^{-c}} - \frac{c^2}{2}\left(\frac{1-(c+1)e^{-c}}{1-ce^{-c}}\right)^2 \right),
\end{equation}
and $\V{\beta(G)}=\Theta(n)$.
\end{theorem}

\textbf{Comparison of the two models.} The plot of the constant term $C$ given by~\eqref{eq:C} as a function of $c$ is shown in Figure \ref{fig:grafic}.
It is interesting to notice that the constant term in the expectation in this latter model is (much) bigger than the constant obtained in Theorem~\ref{thm:main-tree}.
This shows that these two models are qualitatively different.
A possible explanation for this is the following: the second model generates many small trees, for which, relatively to the number of  vertices in the whole graph, a bigger subset is needed to distinguish all vertices (for example, for isolated vertices all of them except one has to be taken, for trees of size 2 and 3 one vertex has to be taken, and in general, the bigger the number of vertices of a tree, the smaller the proportion of vertices that has to be chosen).
Unfortunately, we are not able to calculate the leading constant of the variance in the second model, and thus we cannot compare the two variances.
\begin{figure}[htb]
\begin{center}
\includegraphics[width=12.6cm]{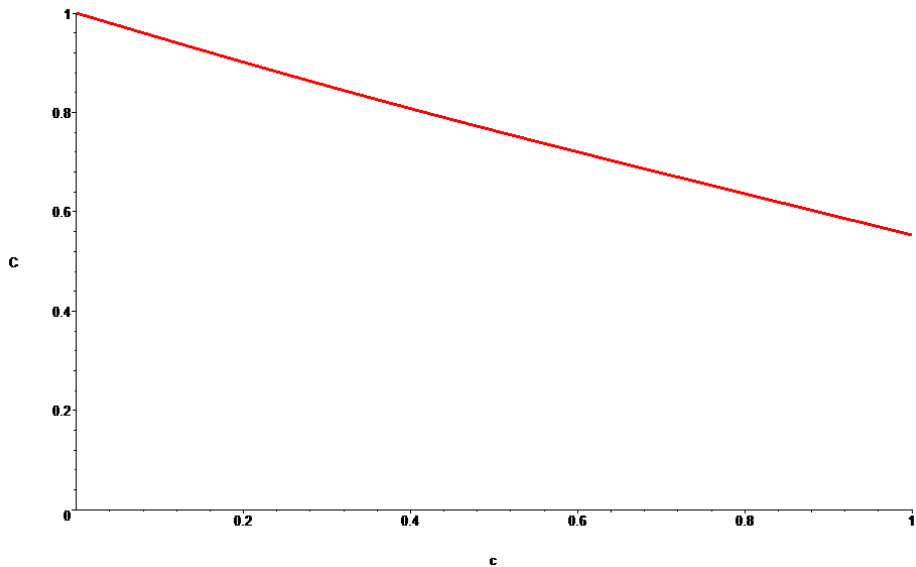}
\end{center}
\caption{The constant term in the mean of the metric dimension when $c$ moves from 0 to $1$. For $c$ approaching $1$, we have $C \simeq 0.55339767$.}\label{fig:grafic}
\end{figure}

The proofs of both results of this paper are based on Slater's characterization of the metric dimension for trees. In Theorem~\ref{thm:main-tree} we use the methodology of the \emph{Analytic Combinatorics} domain (see \cite{FS}), whereas in Theorem~\ref{thm:main-Gnp} we compute first and second moments and use \emph{Stein's Method} to deduce the limiting distribution.
\\\\
\textbf{Organization of the paper.} In Section~\ref{sec:prel} we describe all necessary preliminaries for the proofs of both models. Section~\ref{sec:ran-trees} is then entirely devoted to the proof of Theorem~\ref{thm:main-tree}, and Section~\ref{sec:G(n,p)} deals with the proof of Theorem~\ref{thm:main-Gnp}.

\section{Preliminaries} \label{sec:prel}

In this section we introduce all the techniques we use in the paper, namely the Symbolic Method in Analytic Combinatorics, the results needed for deriving normal limiting distributions in both models, a simple version of Stein's Method and two simple well-known facts about $G(n,p)$. \\
\paragraph{\textbf{The Symbolic Method}} The reference book for all this analysis is~\cite{FS}. All graphs considered in this paper are labelled, namely vertices carry distinguishable labels (for a graph $G$ with $n$ vertices, we may assume that the labels belong to $[n]$). Let $\mA$ be a set of labelled objects, and let $|\cdot|$ be a function from $\mA$ to $\mathbb{N}$. If $a\in \mA$, we say that $|a|$ is the \emph{size} of $a$. A pair $(\mA,|\cdot|)$ is called a \emph{combinatorial class}. We only consider combinatorial classes where the number of elements with a prescribed size is finite. Under this assumption, we define the formal power series $A(x)=\sum_{a\in \mA}\frac{x^{|a|}}{|a|!}=\sum_{n=0}^\infty a_n \frac{x^n}{n!}$, and conversely, $[x^n] A(x)= \frac{a_n}{n!}$. We say that $A(x)$ is the \emph{exponential generating function} associated to the combinatorial class $(\mA,|\cdot|)$. The factorial is used in order to deal with the labels of the combinatorial class.

The \emph{union} $\mA \cup \mB$ of two classes $\mA$ and $\mB$ refers to the disjoint union of the classes (and the corresponding induced size). The \emph{labelled Cartesian product} $\mA \times \mB$ of two classes $\mA$ and $\mB$ is the set of pairs $(a,b)$ where $a\in \mA$ and $b\in \mB$, joint with a redistribution of the labels of both $a$ and $b$. The size of $(a,b)$ is the sum of the sizes of $a$ and $b$. The \emph{sequence} of a set $\mA$ (denoted by $\Seq{\mA}$) is $\{\varepsilon\}\cup\mA \cup (\mA\times\mA) \cup (\mA\times\mA\times \mA) \cup \dots$ ($\varepsilon$ denotes an element in the class of size $0$). The \emph{set} construction $\Set{\mA}$ is $\Seq{\mA}/\sim$, where $(a_1,a_2,\dots,a_r)\sim (\widehat{a}_1,\widehat{a}_2,\dots,\widehat{a}_r)$ when
there exists a permutation of indices $\tau$ in $\{1,\dots,r\}$ such that equality $a_i=\widehat{a}_{\tau(i)}$ holds for all $i$. The \emph{restricted set} construction is equivalent to the previous one but when the Cartesian product has only a fixed number of terms.  Finally, the \emph{composition} of two combinatorial classes $\mA$ and $\mB$ is obtained by substituting each atom of each element of $\mA$ by an element of $\mB$. All these constructions are resumed in Table~\ref{table:symbolic}.
\begin{table}[htb]
\begin{center}
\begin{tabular}{c c|c}
  Construction &  & Generating function \\\hline
  Union & $\mA\cup\mB$ & $A(x)+B(x)$ \\
  Product & $\mA\times\mB$ & $A(x)\cdot B(x)$ \\
  Sequence & $\Seq{\mA}$ & $\left(1-A(x)\right)^{-1}$ \\
  Restricted Set & $\Setd{\mA}$ & $\frac{1}{d!} A(x)^d$ \\
 Set & $\Set{\mA}$ & $\exp(A(x))$ \\
 Composition & $\mA \circ \mB$ & $A(B(x))$ \\
\end{tabular}
\caption{The Symbolic Method. In the table, GFs
associated to classes $\mA$ and $\mB$ are $A(x)$ and $B(x)$,
respectively.}\label{table:symbolic}
\end{center}
\end{table}

The framework of analytic combinatorics is also powerful to handle probabilities in a combinatorial class.
Consider a certain parameter $\chi:\mathcal{A}  \to \mathbb{N}$ on~$\mathcal{A}$.
For $n,m \in \mathbb{N}$, denote by $a_{n,m}$ the number of objects of~$\mathcal{A}$ of size $n$ and parameter $\chi$ equals to $m$.
Define the bivariate generating function
$$
A(x,y) = \sum_{n,m \in \N} \frac{1}{n!} a_{n,m} \, x^n \, y^m,
$$
where $y$ marks the parameter $\chi$.
Observe that $A(x,1)=A(x)$.
For each value of $n$, the parameter $\chi$ defines a random variable $X_{n}$ over elements of $\mathcal{A}$ of size $n$ with discrete probability density function $\mathbb{P}\left(X_n=m\right)=a_{m,n}/{a_n}$.
Hence, this discrete probability distribution can be encapsulated by means of the following expression:
$$
p_n(y) = \frac{[x^n]A(x,y)}{[x^n]A(x,1)}.
$$
The first main theorem of this paper is based on the analysis of such probability distributions.\\
%
%
%
%
%


\paragraph{\textbf{Singularity analysis on bivariate counting formulas}}

By means of complex analytic techniques, it is frequent to obtain functional equations on counting formulas from which we want to extract asymptotic estimates of the coefficients.
Different inversion techniques can be useful for that purpose.
In our work we need to analyze implicit schemes of the form
$$T(x,y) = F(x,y,T(x,y)),$$
for certain analytic functions~$F(x,y,z)$. Under natural conditions on $F$, we can obtain the singular expansion of $T(x,y)$ around its smallest singularity. We rephrase Theorem 2.21 from~\cite{Drmota-trees} (based on the earlier works~\cite{Drmota1, Drmota2, Drmota3}) in a simplified version:
\begin{theorem}[Square-root singularity for implicit equations]\label{eq:sq-single}
Let $F(x,y,z)$ an analytic function around the origin, such that all Taylor coefficients are non-negative, $F(0,y,z)$ is identically equal to the zero function and $F(x,y,0)\neq 0$. Assume that in the region of convergence of $F(x,y,z)$ the system of equations
\begin{equation}\label{eq:syst-fonamental}
z=F(x,1,z),\, 1=\frac{\partial}{\partial z} F(x, 1, z)
\end{equation}
has a non-negative solution $(x,z)=(\rho, \tau)$ such that $\frac{\partial}{\partial x} F(\rho, 1, \tau) \neq 0$ and $\frac{\partial^2}{\partial y^2}F(\rho,1,\tau)\neq 0$. Assume that the counting formula $T(x,y)$ is defined by the implicit scheme $T(x,y) = F(x,y,T(x,y))$. Then, $T(x,y)$ is an analytic function around the origin, with non-negative Taylor coefficients. Additionally, there exist functions $f(y)$, $g(y)$, $h(y)$, $q(y)$ and $\rho(y)$ which are analytic around $x = \rho=\rho(1)$, $y = 1$ such that $T(x,y)$ is analytic for $|x| < \rho$ and $|y-1|<\varepsilon$ (for some $\varepsilon > 0$), and has an expansion of the form
$$
T(x,y) = f(y) + g(y) \left(1-\frac{x}{\rho(y)}\right)^{1/2}+ h(y)\left(1-\frac{x}{\rho(y)}\right)+q(y)\left(1-\frac{x}{\rho(y)}\right)^{3/2}  +O\left(\left(1-\frac{x}{\rho(y)}\right)^2\right),
$$
locally around $x=\rho(y)$.
\end{theorem}

Once we know the singular behavior of a bivariate generating function, we can study, by means of general results, the limiting distribution of the  parameter we are codifying. In this context, the \emph{Quasi-powers Theorem}~\cite{Hwang} gives sufficient conditions to assure normal limiting distributions. In the following simplified version we adapt the hypothesis to the expansions we will find in the analysis:
\begin{theorem}[Quasi-Powers Theorem~\cite{Hwang}]
\label{theo:quasi-powers}
Let $F(x,y)$ be a bivariate analytic function on a neighborhood of $(0,0)$, with non-negative coefficients.
Assume that the function $F(x,y)$ admits, in a region
\[
\mathcal{R}=\{|y-1|<\varepsilon\}\times\{|x|\leq r\}
\]
for some $r,\varepsilon > 0$, a representation of the form
\[
F(x,y) = A(x,y) + B(x,y) \, C(x,y)^{-\alpha},
\]
where $A(x,y)$, $B(x,y)$ and $C(x,y)$ are analytic in $\mathcal{R}$, and such that
\begin{itemize}
\item $C(x,y) = 0$ has a unique simple root $\rho<r$ in $|x|\leq r$,
\item $B(\rho,y) \neq 0$,
\item neither $\partial_x C(\rho,y)$ nor $\partial_y C(\rho,y)$ vanish, so there exists a non-constant function $\rho(y)$ analytic at $y = 1$ such that $\rho(1)=\rho$ and $C(\rho(y),y) = 0$,
\item finally
\[
\sigma^2 = -\frac{\rho''(1)}{\rho(1)}-\frac{\rho'(1)}{\rho(1)}+\left(\frac{\rho'(1)}{\rho(1)}\right)^2
\]
is different from $0$.
\end{itemize}
Then the sequence of random variables with density probability function
\[
p_n(y)=\frac{[x^n]{F(x,y)}}{[x^n]{F(x,1)}}
\]
converges in distribution to a normal distribution.
The corresponding expectation $\mu_n$ and variance $\sigma_n^2$ converge asymptotically to $-\frac{\rho'(1)}{\rho(1)} \, n$ and~$\sigma^2 n$, respectively.
\end{theorem}

\paragraph{\textbf{Stein's Method}}
We also make use of the following theorem, which is an adaptation of Stein's Method for the setting of random graphs (see~\cite{Barbour}):
\begin{theorem}\label{thm:Stein}(Theorem 1 of \cite{Barbour} and its following remarks):
Let $I$ be a finite subset of $\N$, and let $\{X_i\}_{i \in I}$ be a family of (possibly dependent) random variables of zero expectations, and such that $W=\sum_{i \in I} X_i$ has variance $1$. For each $i \in I$, let $K_i \subseteq I$, and define  $Z_i=\sum_{k \in K_i} X_k$ and $W_i=\sum_{k \notin K_i}X_k$ (so that $W=W_i+Z_i$).
Assume that for each $i \in I$, $W_i$ is independent of $X_i$ and $Z_i$, and that $X_i,W_i,Z_i$ have finite second moments. Define
\begin{equation}\label{eq:epsilon}
\varepsilon = 2 \sum_{i \in I} \sum_{k, \ell \in K_i} \left( \E{|X_i X_k X_{\ell}|}+\E{|X_iX_k|}\E{|X_{\ell}|} \right).
\end{equation}
Then, with $\Phi$ denoting the distribution function of the standard normal law,
$$
d(\mathcal{L}(W, \Phi) \leq K\varepsilon
$$
for some universal constant $K$. Hence, if $\{W^{(n)}\}$ is a sequence of random variables, such that each $W^{(n)}$ satisfies the conditions above, and such that the value $\varepsilon^{(n)}$ associated with $W^{(n)}$ converges to $0$ as $n \to \infty$, then $\{W^{(n)}\}$ converges to the standard normal law.
\end{theorem}

\begin{remark}
If one considers the traditional Kolmogorov-Smirnov distance $\delta_n = \sup_x |F_n(x) - \Phi(x)|$ between a distribution function $F_n$ and the standard normal distribution, in general $\delta_n = O(\varepsilon^{1/2})$, and at the cost of greater effort in many cases also $\delta_n=O(\varepsilon)$, see~\cite{Barbour, Chen}.
\end{remark}
\paragraph{\textbf{Properties of the $G(n,p)$ model}}  We also make use of the following two facts about random graphs $G(n,p)$ with $p=\frac{c}{n}$ and $0 < c < 1$.
\begin{lemma}(Corollary 5.11 of~\cite{bol})\label{lem:smallcomp}
 Let $G \in G(n,p)$ with $p=\frac{c}{n}$ and $0 < c < 1$. Then, there exists some $C > 0$ such that with probability at least, say, $1-n^{-5}$, all connected components of $G$ have size at most $C \log n.$
\end{lemma}
\begin{lemma}(Theorem 5.7 of~\cite{bol})\label{lem:notrees}
 Let $G \in G(n,p)$ with $p=\frac{c}{n}$ and $0 < c < 1$ and denote by $Z$ the random variable counting vertices not belonging to trees in $G$. Then $\E{Z}=O(1)$.
\end{lemma}
%

\section{The uniform model}\label{sec:ran-trees}

In this section we study the limiting metric dimension for a random tree chosen uniformly at random among all trees with $n$ vertices. This combinatorial family can be encoded by means of generating functions. Since all trees considered in this paper are labelled (namely, vertices carry distinguishable labels), generating functions are exponential in the vertices and ordinary in the rest of the variables. This step is carried out in Subsection~\ref{subs:enumeration}. Later, by means of asymptotic techniques we prove Theorem~\ref{thm:main-tree} in Subsection~\ref{subs:asympt-analysis}.
%

\subsection{Enumeration} \label{subs:enumeration}

%

\subsubsection{Definitions. Intermediate families of trees} \label{subsect:mobiles}
Given a tree $T$, let us consider the tree $R$ obtained from $T$ by erasing vertices of degree $2$, that is, contracting all paths whose interior vertices are of degree $2$, to a single edge (see Figure~\ref{fig:special}).
We call $R$ the \emph{special tree} associated to $T$, and we observe that $R$ does not have vertices of degree $2$.
Reciprocally, $T$ can be obtained from $R$ by subdividing the edges of $R$.
We denote by $\mathcal{S}$ the family of special trees.

\begin{figure}[htb]
\begin{center}
\includegraphics[width=8.5cm]{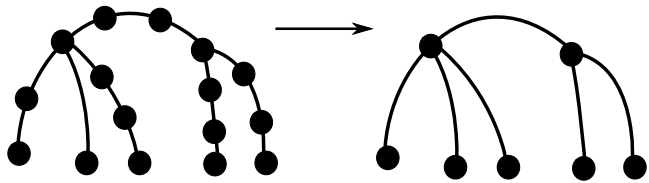}
\end{center}
\caption{A tree (left) and its associated special tree (right).}\label{fig:special}
\end{figure}

Special trees encode all the enumerative information needed to study the metric dimension of random trees and random forests: if $T$ is a tree and $R$ is its associated special tree, then, due to Slater's characterization, $\beta(T)=\beta(R)$.
Moreover, the metric dimension of a special tree, different from a single edge (that could only be obtained when starting from a path), is equal to the number of leaves minus the number of vertices incident to some leaf.
We will exploit this characterization in the next sections.

To study special trees we start with the analysis of an auxiliary family which we name \emph{mobiles}.
The family of mobiles is denoted by $\mathcal{P}$.
Mobiles are rooted trees with a special distinguished half-edge (that we call \emph{leg}) incident with a vertex in the tree which is not a leaf, such that
the degree of each vertex is different from $2$ (the degree of a vertex $v$ is the number of half-edges incident with $v$).
As a special case, the tree with a single vertex incident with a half edge will be also considered inside the family.
We say that the unique vertex incident with the leg is the \emph{root} vertex of the tree.
See Figure~\ref{fig:mobile1} for an example of two mobiles (the leg is represented by an arrow).
\begin{figure}[htb]
\begin{center}
\includegraphics[width=7.0cm]{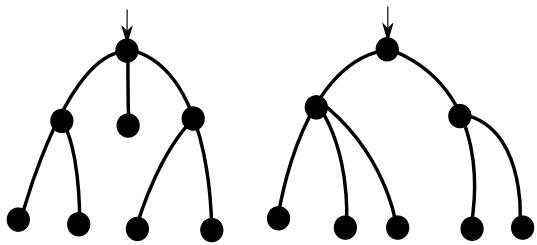}
\end{center}
\caption{Two mobiles. The degrees of the roots are 4 and 3, respectively.}\label{fig:mobile1}
\end{figure}

We use the variable $x$ to encode vertices. The counting formulas considered will be exponential in $x$ (as they are labelled graphs). The other parameters (as the metric dimension) are ordinary: variables $u$, $v$ are used to encode leaves and vertices incident with a leaf. Let
\begin{equation*}\label{eq:rtree-def}
\Tv:=\Tv(x,u,v)=\sum_{n,l,k \geq 0} p_{n,l,k}\frac{x^n}{n!} u^l v^k
\end{equation*}
be the generating function associated to mobiles, where $p_{n,l,k}$ is the number of mobiles with $n$ vertices, $l$ leaves and $k$ internal vertices incident to some leaf. We similarly denote by $S:=S(x,u,v)$, $T:=T(x,u,v)$ and $G:=G(x,u,v)$ the counting formulas for special trees, trees and forests, respectively. Observe that by writing $v^{-1}=u=y$, the variable $y$ encodes the metric dimension in each counting series. We also consider enriched families of rooted trees. In particular, we study families of edge-rooted, edge-oriented rooted and rooted special trees. The corresponding counting formulas are denoted by $S_{\bullet-\bullet}:=S_{\bullet-\bullet}(x,u,v)$, $S_{\bullet\rightarrow\bullet}:=S_{\bullet\rightarrow\bullet}(x,u,v)$ and $S_{\bullet}:=S_{\bullet}(x,u,v)$, respectively. An example of each family is given in Figure~\ref{fig:mobile2}.
\begin{figure}[htb]
\begin{center}
\includegraphics[width=15.0cm]{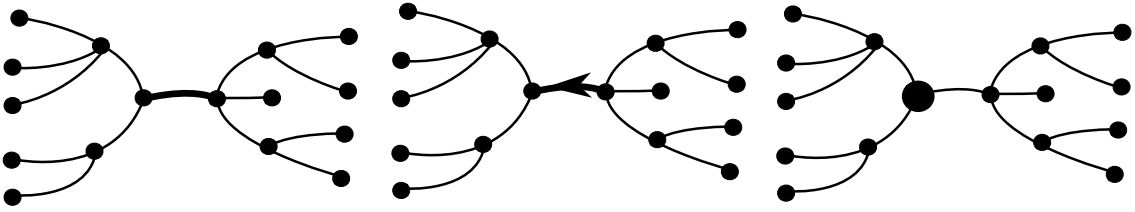}
\end{center}

\caption{Three examples of (edge and vertex) rooted trees.}\label{fig:mobile2}
\end{figure}

The generating function $\Tv$ of mobiles satisfies a recursive description in terms of the degree of the root vertex: a mobile is either a vertex (which is also a leaf, hence it is codified by a term $ux$), or otherwise, the root vertex is incident to a leaf or not. We denote these last two families by $\mathcal{P}_1$ and $\mathcal{P}_2$, and the corresponding counting formulas by $U$, $V$, respectively. For example, the left mobile in Figure~\ref{fig:mobile1} belongs to $\mathcal{P}_1$, and the right one to $\mathcal{P}_2$. In particular:
\begin{equation}\label{eq:mobile1}
\Tv=ux+U+V.
\end{equation}
Let us find relations between $U$, $V$ and $\Tv$. In order to do so, consider the degree of the root vertex of a tree of one of these families. Call it $d+1$. By assumption $d\geq 2$. In the first case, at least one of the pending trees is a leaf, and hence the root vertex must be also codified with $v$. We obtain then the term $\frac{1}{d!}vx\left(\Tv^d-(\Tv-ux)^d\right)$. Finally, summing over  all possible values of $d$ we get
\begin{equation}\label{eq:mobile3}
U=vx\sum_{d\geq 2}\frac{1}{d!}\left(\Tv^d-(\Tv-ux)^d\right)=vx\left(\exp(\Tv)-\exp(\Tv-ux)-ux\right).
\end{equation}
In the second case, the family of mobiles in $\mathcal{P}_2$ whose root vertex is $d+1$ is combinatorially equivalent to
$$\{\bullet\} \times \Setd{\mathcal{P} \setminus \stackrel{\downarrow}{\bullet}},$$
and hence, by the Symbolic Method we have that
\begin{equation} \label{eq:mobile2}
V=x \sum_{d\geq 2} \frac{1}{d!} (\Tv-xu)^d= x\left(\exp(\Tv-ux)-1-\Tv+ux\right).
\end{equation}
Combining \eqref{eq:mobile1}, \eqref{eq:mobile3} and~\eqref{eq:mobile2} we get the following implicit expression for $\Tv$:
\begin{equation}\label{eq:rtree2}
\Tv=(u-1)x+u(1-v)x^2+\left(v+(1-v)\exp(-ux)\right)x\exp(\Tv)-x\Tv.
\end{equation}
Observe that, by writing $u=v=1$ in~\eqref{eq:rtree2}, we recover a slight variation of the classical relation for rooted labelled trees: writing $\Tv(x):=\Tv(x,1,1)$ we have that $\Tv(x)=x\left(\exp(\Tv(x))-\Tv(x)\right)$ (rooted labelled trees without vertices of degree $2$).

\subsubsection{The unrooting argument}\label{subsec:unrooting}
The second step consists in expressing the counting function for rooted special trees (namely, $S_{\bullet\rightarrow\bullet}$ and $S_{\bullet}$) in terms of mobiles. Observe now that
\begin{equation}\label{eq:dys-theorem2}
S(x) = S_{\bullet}(x)-\frac{1}{2}S_{\bullet\rightarrow\bullet}(x)
\end{equation}
because for each tree the number of vertices is by one bigger than the number of edges.

We start by analyzing $S_{\bullet\rightarrow\bullet}$. Observe that cutting the marked edge of an element in $S_{\bullet\rightarrow\bullet}$ determines an ordered pair of rooted trees: each of the resulting two trees has a root (the resulting half edges) and a root vertex (the vertices incident with the initial marked edge). Hence, these two objects are again mobiles. This pair of mobiles is ordered because the root edge is oriented. Hence, the expression for $S_{\bullet\rightarrow\bullet}$ is obtained by combining all possibilities for this pair of trees:
\begin{equation}\label{eq:root-aresta}
S_{\bullet\rightarrow\bullet}=ux^2+2uxU+2uvxV+U^2+V^2+2UV.
\end{equation}
Observe that the term $ux^2$ arises from an oriented edge.

Let us analyze now $S_{\bullet}$. To get its expression in terms of mobiles, we distinguish three cases depending on the degree of the pointed vertex of a tree in $S_{\bullet}$. First, if the degree of the pointed vertex is $0$, then we know that we started from an isolated vertex, hence we have the term $ux$. Second, if the degree of the pointed vertex is equal to 1 (namely, a leaf), we can decompose the counting formula in terms of $U$ and $V$. The strategy is to cut the unique edge incident with the pointed vertex, obtaining the mobile with one vertex (term $ux$) and an arbitrary mobile. This translates in the following way into the generating functions context:
$$ux^2+uxU+uvxV.$$
Now, let us assume that the pointed vertex in our tree in $S_{\bullet}$ has degree greater than 2 (recall that special trees do not have vertices of degree 2). The combinatorial decomposition depends on whether the pointed vertex is incident to a leaf or not. This gives the following counting formula in this situation:
\begin{align*}
x \sum_{d\geq 3} \frac{1}{d!} (\Tv-xu)^d+vx\sum_{d\geq 3}\frac{1}{d!}\left(\Tv^d-(\Tv-ux)^d\right).\\
\end{align*}
Putting all the contributions together gives the following expression for $S_{\bullet}$:
\begin{align}\label{eq:root-vertex}
S_{\bullet}=& ux+ux^2+uxU+uvxV+ (1-v)x \left(\exp(\Tv-ux)-1-(\Tv-ux)-\frac{(\Tv-ux)^2}{2}\right)\\\nonumber
&+ vx\left(\exp(\Tv)-1-\Tv-\frac{\Tv^2}{2}\right).
\end{align}
Applying the final relation (namely, Equation~\eqref{eq:dys-theorem2}), we obtain the desired counting formula. The explicit expression of $S$ in terms of $P$ is long, but it can be deduced immediately from~\eqref{eq:root-aresta} and ~\eqref{eq:root-vertex}. The first terms in the Taylor expansion of $S$ are the following ones:
\begin{align*}
S=& ux+ \frac{1}{2} u x^2+ \frac{1}{6} u^3 v x^4 + \frac{1}{24} u^4 v x^5+  \left(\frac{1}{8} u^4  v^2  + \frac{1}{120} u^5 v \right) x^6 +  \left(\frac{1}{720} u^6 v + \frac{1}{12} u^5  v^2 \right) x^7 \\
&+ \left(\frac{1}{5040} u^7  v + \frac{1}{8} u^5  v^3  + \frac{5}{144} u^6  v^2 \right) x^8+ O(x^9).
\end{align*}
Recall that in this computation trees with vertices of degree 2 are not considered.

\subsubsection{From special trees to trees}\label{subsec:final-trees}

In the last step we can now go back from (unrooted) special trees to (unrooted) trees and forests.  It remains to recover vertices of degree $2$. Observe that general trees are obtained from special trees by substituting each edge by a (possibly empty) sequence of vertices of degree 2. As a tree with $n$ vertices has $n-1$ edges, we need to make the substitution $x^n \leftarrow x^{n} (1-x)^{-n+1}$. Hence, $T=(1-x)S\left(\frac{x}{1-x},u,v\right)$. The first terms in the Taylor expansion of $T$ are the following ones:
\begin{align*}
T =& ux+\frac{1}{2}u x^2 +\frac{1}{2}ux^3 +\left(\frac{1}{2}u+\frac{1}{6}u^3 v\right)x^4+\left(\frac{1}{2}u+\frac{1}{2}u^3v+\frac{1}{24}u^4v\right)x^5\\
&+\left(\frac{1}{2}u+\frac{1}{8}u^4v^2+ u^3 v+ \frac{1}{6}u^4 v+\frac{1}{120}u^5v\right)x^6\\
&+\left(\frac{1}{2}u+\frac{5}{3}u^3 v+\frac{5}{12}u^4 v+\frac{1}{24}u^5 v+\frac{1}{720}u^6 v+\frac{5}{8}u^4 v^2 +\frac{1}{12}u^5 v^2 \right)x^7+O(x^8).
\end{align*}
Notice that the subterms of the form $\frac{1}{2}u x^n$ correspond to paths of length $n$, which are slightly special (their metric dimension is always equal to $1$).

\subsubsection{The final system of equations}\label{subsec:system}

Writing $u=v^{-1}=y$ and collecting all the relations we have obtained so far, we get the desired system of equations. In order to make notation simpler, we write $\Tv(x,y):=\Tv(x,y,y^{-1})$, $U(x,y):=U(x,y,y^{-1})$, and so on:

\begin{equation}\label{eq:system1}
\left\{\begin{array}{cl}
\Tv(x,y) =&  (y-1)x+yx^2\left(1-y^{-1}\right)+\left(y^{-1}+(1-y^{-1})\exp(-yx)\right)x\exp(\Tv(x,y))-x\Tv(x,y) \\
\Tv(x,y) =& yx+U(x,y)+ V(x,y)\\
U(x,y)=& \frac{x}{y}\left(\exp(\Tv(x,y))-\exp(\Tv(x,y)-yx)-yx\right)\\
V(x,y)=&  x\left(\exp(\Tv(x,y)-yx)-1-\Tv(x,y)+yx\right)\\
S_{\bullet\rightarrow\bullet}(x,y) =& yx^2+2yxU(x,y)+2xV(x,y)+U(x,y)^2+V(x,y)^2+2U(x,y)V(x,y)\\
S_{\bullet}(x,y) =& yx+yx^2+yxU(x,y)+xV(x,y)+ y^{-1}x\left(\exp(\Tv(x,y))-1-\Tv(x,y)-\frac{\Tv(x,y)^2}{2}\right)\\
&+ \left(1-y^{-1}\right)x \left(\exp(\Tv(x,y)-yx)-1-(\Tv(x,y)-yx)-\frac{(\Tv(x,y)-yx)^2}{2}\right)\\
S(x,y) =&  S_{\bullet}(x,y)-\frac{1}{2}S_{\bullet\rightarrow\bullet}(x,y)\\
T(x,y)   =&  (1-x)S\left(\frac{x}{1-x},y\right) \\
\end{array}\right.
\end{equation}

It remains to find a last equation concerning forests. This combinatorial class can be defined as the disjoint union of two classes, depending on whether some of the connected components are isolated vertices or not. In the first case, the counting formula arises directly from the set operator (which reads as the exponential function in the generating function context). In the second case we split the contribution into two parts: the contribution of isolated vertices (which is $\frac{1}{y}(\exp(xy)-1)$) and the contribution of  components which are not isolated vertices. Putting these contributions together gives the following expression for the generating function associated to forests:
\begin{equation}\label{eq:forests}
G(x,y)= \exp(T(x,y)-yx)+\frac{1}{y}(\exp(yx)-1)\exp(T(x,y)-yx)
\end{equation}

Resuming, in order to get $T(x,y)$, we first compute  $\Tv(x,y)$ using its implicit definition. Then we can obtain both $U(x,y)$ and $V(x,y)$. This pair of counting formulas, together with $\Tv(x,y)$, defines counting formulas for rooted special families, and in particular, it defines $S(x,y)$. Finally, by a change of variable argument, we can deduce $T(x,y)$, and we finally get $G(x,y)$.

\subsection{Asymptotic analysis} \label{subs:asympt-analysis}

We can now analyze by means of singularity analysis the system of equations \eqref{eq:system1} obtained in Subsection~\ref{subsec:system}. The strategy is the following: we first examine the singular behavior of $\Tv(x,y)$ by means of Theorem~\ref{eq:sq-single}. The singular expansion we obtain is translated to all families of (rooted) special trees. One needs to be more careful when dealing with $S(x,y)$, as the equation has negative coefficients and some terms cancel. Later, we deal with the change of variables that defines $T(x,y)$ in terms of $S(x,y)$ and finally we apply the Quasi-Powers Theorem (Theorem~\ref{theo:quasi-powers}) to get the values of the parameters of the resulting normal limiting distribution. In all this section we write $X(y):= (1-x/\rho(y))^{1/2}$, where $\rho(y)$ is an analytic function at $y=1$ that will be defined below.

We start by analyzing the singular behavior of $\Tv(x,y)$:

\begin{lemma}\label{lemma:sing-P}
Let $\rho(y)$ an analytic function in a neighborhood of the origin satisfying the implicit relation
$$1+\rho(y)= \left(1+\frac{e^{y\rho(y)}-1}{y}\right)\rho(y)e^{1-\rho(y)}.$$
Then, the counting formula $\Tv(x,y)$ has a unique square-root singularity when $y$ varies around $y=1$:
\begin{equation}
\label{eq:sing-P}
\Tv(x,y) = \Tv_0(y) + \Tv_1(y) X(y)  + \Tv_2(y)X(y)^2+ \Tv_3(y) X(y)^3+  O\left(X(y)^4\right),
\end{equation}
uniformly with respect to $y$ for $x$ in a small neighborhood of $\rho(y)$, and with $\Tv_0(y)$, $\Tv_1(y)$, $\Tv_2(y)$, $\Tv_3(y)$ analytic in a neighbourhood of $y=1$. More precisely,
$\rho(1)^{-1}=e-1$. Furthermore, we have that $\rho'(1) \simeq -0.12960268 $ and $\rho''(1) \simeq 0.11039081$.
\end{lemma}
\begin{proof}
Write the first equation in System~\eqref{eq:system1} in the form $\Tv(x,y)=F(x,y,\Tv(x,y))$, where
$$F(x,y,z)=x(y-1)+x^2y\left(1-y^{-1}\right)+\left(y^{-1}+(1-y^{-1})\exp(-xy)\right)x\exp(z)-xz.$$
Observe that we cannot apply directly Theorem~\ref{eq:sq-single} because the Taylor coefficients of $F(x,y,z)$ are both positive and negative. In order to overcome this difficulty, write $\Tv(x,y)=xy+W(x,y)$. Then, $W(x,y)$ satisfies the implicit equation formula
\begin{equation*}
W(x,y)=-(x+x^2)+\left(1+\frac{e^{xy}-1}{y}\right)xe^{W(x,y)}-xW(x,y).
\end{equation*}
Hence, $W(x,y)=H(x,y,W(x,y))$ for the multivariate entire function
$$H(x,y,z)=-(x+x^2)+\left(1+\frac{e^{xy}-1}{y}\right)xe^{z}-xz.$$

Developing the exponential terms in $H(x,y,z)$ it is straightforward to check that the Taylor coefficients of $H$ are non-negative. Additionally, $H(0,y,z)$ is identically equal to 0 and $H(x,y,0) \neq 0$. The system of equations given in~\eqref{eq:syst-fonamental} is the following one:
$$\tau=H(\rho,1,\tau),\, 1=\frac{\partial}{\partial z } H(\rho,1,\tau),$$
which has the solution $\rho=(e-1)^{-1}$, $\tau= \frac{e-2}{e-1}$. Additionally, it is easy to check that $\frac{\partial}{\partial x} H(\rho,1,\tau)\neq 0$ and $\frac{\partial^2}{\partial y^2} H(\rho,1,\tau)\neq 0$. We are then under the assumptions of Theorem~\ref{eq:sq-single}, and $W(x,y)$ has a square-root expansion of the form
$$
W(x,y) = W_0(y) + W_1(y) X(y)+W_2(y)X(y)^2+ W_3(y) X(y)^3+  O\left(X(y)^4\right),
$$
locally around $x=\rho(y)$ (in a neighborhood of $y=1$), with $W_0(y)$, $W_1(y)$, $W_2(y)$, $W_3(y)$ and $\rho(y)$ analytic in this neighborhood, such that $W_0(1)=\frac{e-2}{e-1}$. Finally, because $\Tv(x,y)=xy+W(x,y)$ and $xy$ is an entire function, we conclude that $\Tv(x,y)$ has the square-root expansion
$$
\Tv(x,y) = \Tv_0(y) + \Tv_1(y) X(y)+  \Tv_2(y)X(y)^2+ \Tv_3(y) X(y)^3+  O\left(X(y)^4\right),
$$
with $\Tv_0(y)$, $\Tv_1(y)$,$\Tv_2(y)$ and $\Tv_3(y)$ analytic in a neighborhood of $y=1$, and $\Tv_0(1)=1$.

Let us move to the study of the derivatives of $\rho(y)$ evaluated at $y=1$. For each choice of $y$ in a neighborhood of $1$, the system of equations $\tau=H(x,y,\tau),\, 1=\frac{\partial}{\partial z}H(x,y,\tau)$ has a unique solution $(x,z)=(\rho(y),\Tv(\rho(y),y))$. From this set of equations we deduce that $\rho(y)$ satisfies the implicit formula
\begin{equation*}\label{eq:rho}
1+\rho(y)= \left(1+\frac{e^{y\rho(y)}-1}{y}\right)\rho(y)e^{1-\rho(y)},
\end{equation*}
from which we can deduce (by successive derivatives) exact expressions for both $\rho'(1)$ and $\rho''(1)$. Indeed, expressions for $\rho'(1)$ and $\rho''(1)$ can be computed exactly, but they are long. We only provide exact numerical approximations to these values.
\end{proof}

In fact, we can determine by means of indeterminate coefficients the different functions in $y$ involved in Lemma~\ref{lemma:sing-P}. For example, the first term $\Tv_0(y)$ satisfies the implicit equation $\Tv_0(y)=F(\rho(y),y,\Tv_0(y))$. It is clear that the singular behaviour of $U(x,y)$, $V(x,y)$, $S_{\bullet\rightarrow\bullet}(x,y)$ and $S_{\bullet}(x,y)$ are the same as the one for $\Tv(x,y)$, because the previous equations are analytic transformations of the last counting formula (since we consider $y$ close to $1$, the function $1/y$ is analytic). Indeed, by straightforward computations (i.e., Taylor expansions) one can see that the singular expansions of $U(x,y)$, $V(x,y)$, $S_{\bullet\rightarrow\bullet}(x,y)$ and $S_{\bullet}(x,y)$ are also of square-root type.

However, this is not the case when dealing with $S(x,y)$: expanding $S(x,y)$ around $x=\rho(y)$ we obtain the singular expansion
$$
S(x,y) = S_0(y) +  S_2(y)X(y)^2+ S_3(y) X(y)^3+  O\left(X(y)^4\right)
$$
for certain functions $S_0(y)$, $S_2(y)$ and $S_3(y)$ analytic at $y=1$.
In other words, the corresponding term $S_1(y)$ vanishes in a neighborhood of $y=1$.
This fact can be argued analytically in the following way: since $S_{\bullet}(x,y)=x \frac{\partial}{\partial x}S(x,y)$, we have
$$S(x,y)=\int_{0}^{x}\frac{S_{\bullet}(s,y)}{s} ds$$
This integral gives the previous result because $S_{\bullet}(x,y)$ has an square-root expansion around $x=\rho(y)$.
This argument is the analytic counterpart of the unrooting argument which arises from the relation $S(x,y) =  S_{\bullet}(x,y)-\frac{1}{2}S_{\bullet\rightarrow\bullet}(x,y)$.

The last step needed is to obtain $T(x,y)$ from $S(x,y)$. We encapsulate this analysis in a lemma.

\begin{lemma}\label{lemma:sing-T}
Let
\begin{equation*}\label{eq:afegir2}
R(y)=\frac{\rho(y)}{1+\rho(y)},
\end{equation*}
where $\rho(y)$ is the function defined in Lemma~\ref{lemma:sing-P}, and let $\overline{X}(y)=\sqrt{1-x/R(y)}$. Let $y$ be a positive real number in an small neighbourhood of $1$. Then the generating function $\Tv(x,y)$ has a unique singularity at $x=R(y)$, and it admits a singular expansion at this point of the form
\begin{equation*}
\label{eq:sing-P}
T(x,y) = T_0(y) + T_2(y)\overline{X}(y)^2+ T_3(y) \overline{X}(y)^3+  O\left(\overline{X}(y)^4\right),
\end{equation*}
uniformly with respect to $y$ for $x$ in a small neighborhood of $R(y)$, and with $T_0(y)$, $T_2(y)$ and $T_3(y)$ analytic at $y=1$. More precisely,
$R(1)^{-1}=e^{-1}\simeq 0.36787944 $. Furthermore, we have that $R'(1) \simeq -0.05178617 $ and $R''(1) \simeq 0.03562445$.
\end{lemma}

\begin{proof}
For each choice of $y$ in a neighborhood of $1$, the smallest real positive singularity of $S(x,y)$ is located at $x=\rho(y)$. Hence, for this value of $y$, $T(x,y)$ has a unique smallest real singularity at $R(y)$, such that $R(y)$ satisfies the condition
\begin{equation*}\label{eq:rho-R}
\rho(y)=\frac{R(y)}{1-R(y)}.
\end{equation*}
The map $f(z)=\frac{z}{1-z}$ is holomorphic in all points $z\neq 1$. Thus, the singular expansion of $S(x,y)$ around $x=\rho(y)$ and $y=1$ is translated directly into the singular expansion of $T(x,y)$ around $x=R(y)$ and $y=1$. Finally, we obtain expressions for $R'(y)$ and $R''(y)$ from the equation satisfied by $\rho(y)$ claimed in Lemma~\ref{lemma:sing-P}. Joint with the estimates for $\rho'(1)$ and $\rho''(1)$ obtained in Lemma \ref{lemma:sing-P}, the estimates claimed for $R'(1)$ and $R''(1)$ hold.
\end{proof}

Now, Theorem~\ref{thm:main-tree} is an easy consequence of Lemma~\ref{lemma:sing-T}: the expansion of $T(x,y)$ around its smallest singularity satisfies the assumptions of Theorem~\eqref{theo:quasi-powers}. Additionally,
$$
\mu=-\frac{R'(1)}{R(1)}\simeq 0.1407694113 ,\,\,\sigma^2 = -\frac{R''(1)}{R(1)}-\frac{R'(1)}{R(1)}+\left(\frac{R'(1)}{R(1)}\right)^2\simeq 0.06374815134,
$$
and Theorem~\ref{thm:main-tree} follows by observing $\sigma^2\neq 0$ and applying the Quasi-Powers Theorem. Since the counting formula for forests is an analytic transform of the generating function for trees (see Equation~\eqref{eq:forests}, and recall that the exponential function is an analytic transform), the same result holds in random forests. Hence, the limit law corresponding to $F_n$ as it is stated in Theorem~\ref{thm:main-tree} holds.

\section{The $G(n,p)$ model: proof of Theorem~\ref{thm:main-Gnp}}\label{sec:G(n,p)}

The proof  of (i) in Theorem~\ref{thm:main-Gnp} is a straightforward calculation, and we give it here for the sake of completeness.
Denote by $I$ the random variable counting the number of isolated vertices and by $N$ the random variable counting non-isolated vertices.
Then, for $p \in o\left(n^{-1}\right)$, the probability that a fixed vertex is isolated is equal to $(1-p)^{n-1}$.
Hence, $\E{I}=n(1-p)^{n-1} = n(1+o(1))$ and finally, $\E{N}=o(n)$.
By Markov's inequality, $N=o(n)$ a.a.s., and hence $I=n(1+o(1))$ a.a.s.
As every isolated vertex except for one has to be taken into a resolving set, $\beta(G)=n(1+o(1))$ a.a.s., and (i) of Theorem~\ref{thm:main-Gnp} follows.

Now, consider $G(n,p)$ with $p=\frac{c}{n}$ for some constant $0 < c < 1$. We first give an overview over the proof: for the expected value of $\beta(G)$, we recall standard results about the component structure and degree distribution of vertices. In order to apply Slater's characterization, we find the expected number of vertices of degree at least $3$ that are either adjacent to a leaf or is connected to a leaf via a path of degree $2$ vertices. Note that in this section, in contrast to the previous one, we cannot simply leave out chains of degree $2$ vertices and then subdivide edges, but we rather have to compute the probability of having chains of certain lengths. Next, in order to compute the variance we compute all possible joint second moments of the graph-theoretic concepts appearing in the expectation. The details are lengthy, but the idea is simple: we show that the joint expectations are up to smaller order terms as the product of the expectations. The result will then follow by applying Stein's method.

We first compute the expectation of the metric dimension of a random graph in this model.
\begin{lemma}\label{lem:expGnp}
$$
\E{\beta(G)}=ne^{-c}\left(1+c-\sum_{k\geq 3}\frac{c^k}{k!}\left(1-\left(\frac{1-(c+1)e^{-c}}{1-ce^{-c}}\right)^k\right)-\sum_{k\geq 2}\frac{1}{2}c^{k-1}e^{-(k-1)c}\right)(1+o(1)).
$$
\end{lemma}
\begin{proof}
For a fixed vertex $v$, denote by $X_v$ the random variable counting its vertex degree. We have
$$
\Prob(X_v=1)=(n-1)p(1-p)^{n-2} = ce^{-pn}(1+o(1))=ce^{-c}\left(1+O\left(1/n\right)\right),
$$
and in general, for any $k$,
$$
\Prob(X_v=k)=\binom{n-1}{k}p^k(1-p)^{n-k-1}=\frac{c^k}{k!}e^{-c}\left(1+O\left(k/n\right)\right).
$$
Hence, denoting by $L$ the number of leaves, we obtain $\E{L}=nce^{-c}\left(1+O\left(1/n\right)\right)$. Similarly, denoting by $D_k$ the number of vertices of degree $k$, we obtain $\E{D_k}=n\frac{c^k}{k!}e^{-c}\left(1+O\left(k/n\right)\right)$ and for $k \in \omega(\log n)$, by Lemma~\ref{lem:smallcomp}, $D_k=0$ a.a.s. Also, $\E{I}=ne^{-c}\left(1+O\left(1/n\right)\right)$.

Next, denote by $T_k$ the number of connected components in a random graph $G(n,p)$ that are trees of size $k \geq 2$, and by $P_k$ the number of paths of size $k \geq 2$. Recall that the number of labelled trees of size $k$ is equal to $k^{k-2}$. Observing that in $G(n,p)$ all $k^{k-2}$ labelled trees on $k$ vertices are equally likely to appear, and since there are $k!/2$ labelled paths on $k$ vertices, we have for $k \in O(\log n)$,
$$\E{T_k}=nk^{k-2}\frac{c^{k-1}}{k!}e^{-kc}(1+o(1)),\,\,\E{P_k}=n\frac{c^{k-1}}{2}e^{-kc}(1+o(1))$$
and by Lemma~\ref{lem:smallcomp}, a.a.s., for all $k \in \omega(\log n)$, $T_k$ and $P_k$ are $0$.
Using Stirling's formula, we obtain
$$
\E{T_k} = n \frac{ (ce)^k e^{-kc}}{ c k^2 \sqrt{2 \pi k}}(1+o(1))=n \frac{(ce^{1-c})^k}{c k^2 \sqrt{2 \pi k}}(1+o(1)),
$$
and since for $c < 1$, we have $ce^{1-c} < 1$, the expected number of trees decreases exponentially in $k$, and we obtain
\begin{equation}\label{eq:expdecrease}
\E{T_k} =  n\alpha_0^k (1+o(1))
\end{equation}
for some $0 < \alpha_0 < 1$. Since by Lemma~\ref{lem:notrees} there are in expectation only $O(1)$ vertices which do not belong to trees, the same result holds for component sizes in general. Since for any $c < 1$, clearly  $ce^{-c} < 1$, we also have
\begin{equation}\label{eq:expdecreaseP}
\E{P_k} = n\alpha_1^k (1+o(1))
\end{equation}
for some $0 < \alpha_1 < 1$, and the same holds then also for the number of paths. Since the number of vertices in trees of size $k$ is equal to $kT_k$, also this number decreases exponentially in $k$, with again a different $0 < \alpha_2 < 1$. Finally, once more with some different $0 < \alpha_3 < 1$, the same holds for the number of  pairs of vertices $R_k$ belonging to the same tree of size $k$, since $R_k=\binom{k}{2}T_k.$   In particular, denoting by $R=\sum_{k \geq 2} R_k$ the random variable counting all pairs of vertices belonging to the same connected component,
\begin{equation}\label{eq:nosamecomp}
\E{R} = (1+o(1))n \sum_{k \geq 2} \alpha_3^k = O(n).
\end{equation}

Now, in order to apply Slater's characterization for trees, we define two special concepts: call a vertex $v$ to be \emph{thin}, if it is in a tree component and if it is either a leaf or if it  is of degree $2$ adjacent to another vertex that is thin. Call a vertex $w$ \emph{important} of degree $k$, if it is in a tree component, if it has degree $k \geq 3$ and if it has at least one thin neighbor. See Figure~\ref{fig:thin} for an example of a tree where thin and important vertices are shown. Observe that the set of all important vertices of degree $k \geq 3$ is exactly the set $K$ in Slater's characterization, and thus, for a tree different from a path, its expected metric dimension can be calculated by subtracting the number of all such important vertices from the number of leaves.
\begin{figure}[htb]
\begin{center}
\includegraphics[width=3.8 cm]{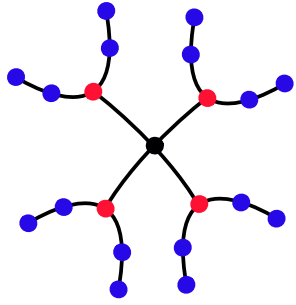}
\end{center}
\caption{A tree whose thin vertices are blue and important vertices are red}\label{fig:thin}
\end{figure}

For a given vertex $w$, expose the edges and non-edges incident to $w$  and suppose that $w$ has degree $k$ with neighbors $v_1,\ldots,v_k$ for some $k \geq 3$. We may assume that $k \in O(\log n)$.  Call another possible neighbor of $v_1$ to be $u$ (different from $v_2,\ldots,v_k$). Then
\begin{equation}\label{eq:thin1}
\Prob(v_1 \mbox{ thin})=\Prob(N(v_1) \cap (V \setminus \{w\})=\emptyset)+\Prob(N(v_1) \cap (V \setminus \{w\})=\{u\}) \Prob(u \mbox { thin } | \; \{u,v_1\} \in E).
\end{equation}
We have
\begin{equation}\label{eq:thin2}
\Prob(N(v_1) \cap (V \setminus \{w\})=\emptyset)=e^{-c}\left(1+O\left(1/n \right)\right),
\end{equation}
and
\begin{equation}\label{eq:thin3}
\Prob(N(v_1) \cap (V \setminus \{w\})=\{u\} ) =ce^{-c}\left(1+O\left(k/n \right)\right),
\end{equation}
since $u$ has to be different from $v_2,\ldots,v_k$. Observe also that $\Prob(N(v_1) \cap (V \setminus \{w\})=\emptyset)=\Theta(1)$ and
$\Prob(|N(v_1) \cap (V \setminus \{w\})|=2)=\Theta(1)$, and thus
\begin{equation}\label{eq:constprob}
\Prob(v_1 \mbox{ thin})=\Theta(1), \: \: \Prob(v_1 \mbox{ not thin})=\Theta(1).
\end{equation}
Now,
$$
\Prob(N(u) \cap (V \setminus \{w,v_1\})=\emptyset )=\Prob(N(v_1) \cap (V \setminus \{w\})=\emptyset)\left(1+O\left(1/n \right)\right)
$$
and, by expanding the recursion defined by~\eqref{eq:thin1} term by term, we see that
\begin{equation}\label{eq:expand}
\Prob(v_1 \mbox{ thin})-
\Prob(u \mbox{ thin } | \; \{u,v_1\} \in E ) \leq e^{-c}O(1/n)+e^{-c}\sum_{j \geq 1} (ce^{-c})^j \left(1+O(k/n)\right)^j \left(O(j/n )\right)^{2j+1}.
\end{equation}
Since the probability to have paths of length $\omega(\log n)$ is smaller than $n^{-5}$, say, the contribution of all terms $j \geq C \log n$ is at most $n^{-5}$, we can here and below safely ignore these terms to conclude that
$$
\Prob(v_1 \mbox{ thin})-
\Prob(u \mbox{ thin } | \; \{u,v_1\} \in E )=O(1/n),
$$
and by~\eqref{eq:constprob},
$$
\Prob(u \mbox{ thin } | \; \{u,v_1\} \in E )=\Prob(v_1 \mbox{ thin})(1+O(1/n)).
$$
Thus, plugging this into~\eqref{eq:thin1}, by~\eqref{eq:thin2} and~\eqref{eq:thin3},
$$
\Prob(v_1 \mbox{ thin}) = \frac{e^{-c}}{1-ce^{-c}} \left(1+O\left(k/n\right)\right)
$$
and
$$
\Prob(v_1 \mbox{ not thin}) = 1-\frac{e^{-c}}{1-ce^{-c}} \left(1+O\left(k/n\right)\right).
$$
Next, denote by $\mathcal{P}_{\ell}$ the event that a path of length $\ell \geq 0$ is attached to $v$ (not going through $w$), so that we can write
\begin{equation}\label{eq:decompose}
\Prob\left(v_2 \mbox{  thin} \; | \; v_1 \mbox{ thin}\right) = \sum_{\ell \geq 0}
\Prob\left(v_2 \mbox{  thin} \; | \; v_1 \mbox{ thin} \wedge \mathcal{P}_{\ell}\right) \; \;
\Prob\left( \mathcal{P}_{\ell} \; | \; v_1 \mbox{ thin}\right).
\end{equation}
We have
$$
\Prob\left( \mathcal{P}_{\ell} \; | \; v_1 \mbox{ thin}\right)=\Prob(\mathcal{P}_{\ell}) / \Prob(v_1 \mbox{ thin}),
$$
since $\Prob( v_1 \mbox{ thin } | \; \mathcal{P}_{\ell})=1$. Using~\eqref{eq:expdecreaseP} and~\eqref{eq:constprob}, we see
that
\begin{equation*}\label{eq:thin4}
\Prob\left( \mathcal{P}_{\ell} \; | \; v_1 \mbox{ thin}\right)=\alpha^{\ell+1}
\end{equation*}
for some $0 < \alpha < 1$. For the term $\Prob\left(v_2 \mbox{  thin} \; | \; v_1 \mbox{ thin} \wedge \mathcal{P}_{\ell}\right)$ note that $\ell$ additional vertices and its incident edges and non-edges have been exposed, giving an additional correction term of $O(\ell/n).$ Expanding then the recursive formula term by term as in~\eqref{eq:expand},
\begin{eqnarray*}\label{eq:expand2}
\Prob(v_1 \mbox{ thin})
- \Prob\left(v_2 \mbox{  thin} \; | \; v_1 \mbox{ thin} \wedge \mathcal{P}_{\ell}\right)
 \leq  e^{-c}O(\ell/n)+e^{-c}\sum_{j \geq 1} (ce^{-c})^j \left(1+O(k/n)\right)^j \left(O((\ell+j)/n )\right)^{2j+1},
\end{eqnarray*}
and hence, as before,
$$
\Prob\left(v_2 \mbox{  thin} \; | \; v_1 \mbox{ thin} \wedge \mathcal{P}_{\ell}\right) = \Prob(v_1 \mbox{ thin}) (1+O(\ell/n)).
$$
Thus, for the same $0 < \alpha < 1$ as above, \eqref{eq:decompose} gives
$$
\Prob\left(v_2 \mbox{  thin} \; | \; v_1 \mbox{ thin}\right)=  \Prob(v_1 \mbox{ thin}) \sum_{\ell \geq 0} (1+O(\ell/n)) \alpha^{\ell+1}=\Prob(v_1 \mbox{ thin})(1+O(1/n)),
$$
since $\sum_{\ell \geq 0}\Prob\left( \mathcal{P}_{\ell} \; | \; v_1 \mbox{ thin}\right)=\sum_{\ell \geq 0} \alpha^{\ell+1}=1.$
Hence,
$$
\Prob(v_1 \mbox{ thin} \; \wedge v_2 \mbox { thin})= \Prob\left(v_1 \mbox { thin}\right)^2 (1+O(1/n)).
$$
By~\eqref{eq:constprob}, we also have
\begin{eqnarray*}
\Prob\left(v_2 \mbox{  not thin} \; | \; v_1 \mbox{ thin} \wedge \mathcal{P}_{\ell}\right)&=&\Prob( v_1 \mbox { not thin})(1+O(\ell/n)) \\
\Prob\left(v_2 \mbox{ not thin} \; | \; v_1 \mbox{ not thin} \wedge \mathcal{P}_{\ell}\right)&=&\Prob( v_1 \mbox { not thin})(1+O(\ell/n))
\end{eqnarray*}
and thus in particular also
$$
\Prob(v_1 \mbox{ not thin} \; \wedge v_2 \mbox { not thin})= \Prob\left(v_1 \mbox { not thin}\right)^2 (1+O(1/n)).
$$
Denoting by $\mathcal{P}_{\ell}^i$ the event that vertex $v_i$ has a path of length $\ell$ attached to it (not going through $w$), expanding as in~\eqref{eq:expand}, we also have
$$	
\Prob(\mathcal{P}_{\ell_j}^j \; | \; \mathcal{P}_{\ell_1}^1 \wedge \ldots \wedge \mathcal{P}_{\ell_{j-1}}^{j-1}) = \Prob(\mathcal{P}_{\ell_j}^j) (1+O((\ell_1+\ldots+\ell_{j-1})/n)),
$$
since all vertices of previous paths and the incident edges are already exposed. Also, from
$$
\Prob(v_j \mbox{ thin } | \; v_1 \mbox{ thin } \wedge \ldots  \wedge v_{j-1} \mbox { thin})
$$ we get an additional error term of at most $(1+O((\ell_1+\ldots+\ell_{j-1})/n))$, and the same additional error term comes from
$$
\Prob(v_j \mbox{ thin } | \; v_1 \mbox{ thin } \wedge \ldots  \wedge v_{j-1} \mbox { thin } \wedge \mathcal{P}_{\ell_1}^1 \wedge \ldots \wedge \mathcal{P}_{\ell_{j-1}}^{j-1}).
$$
Hence, the cumulative error term for $$\Prob(v_j \mbox{ thin } | \; v_1 \mbox{ thin } \wedge \ldots  \wedge v_{j-1} \mbox { thin})$$ is of order at most $\left(1+O((\ell_1+\ldots+\ell_{j-1})/n)\right)^{j^2}$. Proceeding inductively as  in~\eqref{eq:decompose} and thereafter, we obtain thus  for the same $0 < \alpha < 1$ as above,
\begin{eqnarray*}
\Prob(v_j \mbox{ thin } | \; v_1 \mbox{ thin } \wedge \ldots  \wedge v_{j-1} \mbox { thin}) = \Prob(v_1 \mbox{ thin})\sum_{\ell_1 \geq 0} \ldots \sum_{\ell_{j-1} \geq 0} \alpha^{\sum_{i=1}^{j-1} \ell_{i}+1} \left(1+O((\ell_1+\ldots+\ell_{j-1})/n)\right)^{j^2}, \\
\end{eqnarray*}
and also inductively, as before
\begin{eqnarray*}
\Prob(v_j \mbox{ thin } | \; v_1 \mbox{ thin } \wedge \ldots  \wedge v_{j-1} \mbox { thin})&=&  \Prob(v_1 \mbox{ thin})(1+O(j^2/n)) \\
\Prob(v_j \mbox{ not thin } | \; v_1 \mbox{ not thin } \wedge \ldots  \wedge v_{j-1} \mbox { not thin})&=&  \Prob(v_1 \mbox{ not thin})(1+O(j^2/n)).
\end{eqnarray*}
Therefore,
$$
\Prob(v_1 \mbox{ not thin} \wedge \ldots \wedge v_k \mbox{ not thin})=\Prob(v_1 \mbox{ not thin})^k (1+O(k^3/n))
$$
Hence, for a vertex $w$ of degree $k$, we have
\begin{eqnarray*}
\Prob(w \mbox{ important})&=&1-\Prob(v_1 \mbox{ not thin} \wedge \ldots \wedge v_k \mbox{ not thin}) \\
&=&\left(1-\left(1-\left(1+O\left(k^3/n\right)\right)\frac{e^{-c}}{1-ce^{-c}}\right)^k\right) \\
&=&\left(1-\left(\frac{\left(1-(c+1)e^{-c}\right)\left((1+O(k^3/n))e^{-c}\right)}{1-ce^{-c}}\right)^k\right) \\
&=&1-\left(\frac{1-(c+1)e^{-c}}{1-ce^{-c}}\right)^k(1+o(1)) \\
&=&\left(1-\left(\frac{1-(c+1)e^{-c}}{1-ce^{-c}}\right)^k\right)(1+o(1)),
\end{eqnarray*}
and thus, denoting by $J_k$ the number of important vertices of degree $k$, we have
\begin{eqnarray*}
\E{J_k}&=& \E{D_k} \Prob(w \mbox{ important} \; | \; w \mbox{ has degree }k) \\
&=&\E{D_k}\left(1-\left(\frac{1-(c+1)e^{-c}}{1-ce^{-c}}\right)^k\right)(1+o(1)) \\
&=&\frac{nc^k}{k!}e^{-c}\left(1-\left(\frac{1-(c+1)e^{-c}}{1-ce^{-c}}\right)^k\right)(1+o(1)).
\end{eqnarray*}
By Slater's characterization, the metric dimension of a tree is the number of leaves minus the number of important vertices, except for the case of a path, in which case the metric dimension is only one, although there are two leaves and no important vertex. Denote by $Z$ the metric dimension of the connected components which are not trees. Then we have
$$
\E{\beta(G)}=\E{I}+\E{L}-\sum_{k \geq 3} \E{J_k}-\sum_{k \geq 2} \E{P_k}+\E{Z}. $$

By Lemma~\ref{lem:notrees} in expectation there are only $O(1)$ vertices in connected components which are not trees, and hence $\E{Z}=O(1).$ Thus, by the previous results,
$$
\E{\beta(G)}=ne^{-c}\left(1+c-\sum_{k\geq 3}\frac{c^k}{k!}\left(1-\left(\frac{1-(c+1)e^{-c}}{1-ce^{-c}}\right)^k\right)-\sum_{k\geq 2}\frac{1}{2}c^{k-1}e^{-(k-1)c}\right)(1+o(1)),
$$
and the lemma follows.
\end{proof}
Note that the constant given by Lemma~\ref{lem:expGnp}  coincides with the closed formula of $C$ in the statement of Theorem~\ref{thm:main-Gnp}, and the first part of this theorem is proven.
We move on to calculate the variance and will show the following result.
\begin{lemma}\label{lem:var}
$\V{\beta(G)}=\Theta(n).$
\end{lemma}
\begin{proof}
First observe that since there is a constant probability of having a linear number of trees of size $4$, say,  in each connected component there is a constant probability to have metric dimension $1$ or $2$, and different connected components are independent, it is clear that $\V{\beta(G)}=\Omega(n)$. We now show that
$$\E{\beta(G)^2}=\E{\left(I+L-\sum_{k \geq 3}J_k - \sum_{k \geq 2}P_k + Z \right)^2}=(\E{\beta(G)})^2 (1+O(1/n)),$$
implying thus also $\V{\beta(G)}=O(n)$.

Define by $I_v$ the indicator variable which is $1$ if the vertex $v$ is isolated, and $0$ otherwise. Hence $I=\sum_{v \in V} I_v$. Observe that
\begin{align*}
\E{I^2}&=(\E{I})^2(1+O(1/n)) \\
 \E{IL}&=\E{I}\E{L}(1+O(1/n)) \\
  \E{I P_k}&=\E{I}\E{P_k}(1+O(1/n))
  \end{align*} and also
  $$\E{IJ_k}=\E{I}\E{J_k}(1+O(1/n))$$ for any $k$, as one isolated vertex still leaves total freedom on the remaining $n-1$ vertices.

Define furthermore by $L_v$ to be the indicator variable which is $1$ if the vertex $v$ is a leaf, and $0$ otherwise. Note that $L=\sum_{v \in V} L_v$. Considering all pairs of leaves, and distinguishing upon the fact whether they are either connected by an edge, share the same neighbor, or do not share the same neighbor, we have
\begin{eqnarray*}
\E{L^2}&=&\sum_{v \in V} \E{L_v}+\sum_{v \neq w} \E{L_v L_w} \\
&=&\E{L}+n(n-1)\left( p(1-p)^{2n-4}+(n-2)p(1-p)^{2n-5}p+(n-2)(n-3) p^2 (1-p)^{2n-5}\right) \\
&=& (\E{L})^2 (1+O(1/n)).
\end{eqnarray*}
Next, denote by $P_{v_1,\ldots,v_k}$ the indicator variable which is $1$ if the vertices $v_1,\ldots,v_k$ form a connected component which is a path of length $k$ in this order, and $0$ otherwise. Note that $P_k = \sum P_{v_1,\ldots,v_k}$, where the sum is taken over all $k$-tuples of different vertices with the property that the index of $v_1$ is smaller than the index of $v_k$ (recall that there are a total of $\frac{k!}{2}$ labelled paths of length $k$). Then, the contribution to $$\E{LP_k}=\E{\left( \sum_v {L_v} \sum_{v_1,\ldots,v_k} {P_{v_1,\ldots,v_k}}\right)} $$ either comes from a path $P_k$, where $v=v_1$ or $v=v_k$, or from two different connected components. The first term  gives the contribution $\E{P_k}$, and thus this term gives at most $O(n)$ after summing over all $k$.

For the second term, since the contribution comes from different connected components, the random variables are independent, and the only error term comes from the fact that vertices forming part of the connected component of $v$ are excluded from the consideration. More formally, $$\E{LP_k}=\E{L}\E{P_k}(1+O(\ell/n)),$$ where $\ell$ denotes the size of the connected component the leaf $v$ belongs to. Call $u$ the only vertex $v$ is adjacent to and observe the following: conditioning under the fact that $v$ is a leaf in a tree, all trees pending at $u$ in the graph excluding $v$ are still possible, and their occurrences follow the same probability distribution as the trees $T_k$ in the original graph. In particular, the results from~\eqref{eq:expdecreaseP} apply and thus, conditional under the fact that $\ell \geq 2$, the probability of having a connected component of size $\ell \geq 2$ decreases exponentially in $\ell$. Therefore, denoting by $C_v^{\ell}$ the binary random variable yielding $1$ if the connected component $v$ belongs to has size $\ell$, given that $v$ is a leaf, we have for some $0 < \alpha < 1$
\begin{eqnarray*}
\E{LP_k} &= &\sum_{\ell \geq 2} \E{L}\E{P_k}\Prob{\left(C_v^{\ell}=1\right)}(1+O(\ell/n)) \\
&= & \E{L}\E{P_k} \sum_{\ell \geq 2} \alpha^{\ell}(1+O(\ell/n)) \\
&=&  \E{L}\E{P_k} \left( 1 + O(1/n) \right),
\end{eqnarray*}
where the last line follows from the fact that $\sum_{\ell \geq 2}\Prob(C_v^{\ell}=1)=\sum_{\ell \geq 2} \alpha^{\ell}=1$, and also from the fact that $\sum_{\ell \geq 2} \alpha^{\ell}O(\ell/n)=O(1/n).$

Similarly, for $\E{L J_k}$, first recall that by~\eqref{eq:nosamecomp},  in expectation there are only $O(n)$ pairs of vertices belonging to the same connected component, and even when summing over all $k$, we may safely discard them.
%
Otherwise, the contribution comes from two different connected components. The events are independent, and the error term comes from the size $\ell$ of the connected component the leaf $v$ belongs to. Using the notation and the argument as in the analysis of $\E{LP_k}$,
\begin{eqnarray*}
\E{LJ_k} &= &\sum_{\ell \geq 2} \E{L}\E{J_k}\Prob{\left(C_v^{\ell}=1\right)}(1+O(\ell/n)) \\
&=&  \E{L}\E{J_k} \left( 1 + O(1/n) \right).
\end{eqnarray*}

Furthermore, $$\E{P_k^2}=\E{P_k}+\sum \E{P_{v_1,\ldots,v_k} P_{w_1,\ldots,w_k}},$$ where the sum is over all pairs of $k$-tuples which have no vertex in common. Thus, $$\E{P_k^2}=\E{P_k}+(\E{P_k})^2\left(1+O\left(k/n\right)\right).$$  In fact, $$\sum_{k \geq 2} \E{P_k^2}=O(n)+\sum_{k \geq 2} (\E{P_k})^2 \left(1+O\left(k/n\right)\right).$$  For paths of different lengths, the contributions must come from different connected components, and thus, by the same argument we also have for $k < \ell$, $$\E{P_k P_{\ell}}=\E{P_k}\E{P_{\ell}}(1+O\left(k/n\right)).$$ The same argument also shows that for any $k \geq 2$, $\ell \geq 3$, we have $$\E{P_k J_{\ell}}=\E{P_k}\E{J_{\ell}}(1+O\left(k/n\right)).$$ In all cases, even when summing over all $k$ and $\ell$ the contribution of pairs of vertices coming from the same connected component is $O(n)$.

Next, consider $\E{J_k J_{\ell}}$ for $k \geq \ell \geq 3$. As in the analysis of $\E{LP_k}$, by~\eqref{eq:nosamecomp}, pairs of vertices belonging to the same connected component may be disregarded, since in expectation there are only $O(n)$ of them, even when summed over all $k$ and $\ell$.
For pairs of vertices coming from different connected components, observe that the two events are independent. Moreover, we now show that conditioning under the fact that a vertex is important of degree $\ell$, the size of its connected component decreases exponentially, given that it is at least $\ell+1$:
indeed, knowing that a vertex $w$ is important of degree $\ell$ with neighbors $v_1,\ldots,v_{\ell}$ forbids those trees, where each of the vertices $v_1,\ldots,v_{\ell}$ has $2$ or more neighbors other than $w$. Therefore, all trees with at most $3\ell$ vertices are still allowed, and only from then on some trees are forbidden.

Observe the following: once a labelled tree on $m$ vertices, say with labels in $[m]$, is forbidden, by adding a new vertex, say with label $m+1$, to any of the vertices of the tree different from $w$  the tree remains forbidden. On the other hand, a tree which is still allowed may become forbidden by  adding a new vertex. Moreover, for each labelled tree on $m$ vertices, one can always obtain $m-1$ different trees by adding one new vertex with label $m+1$: attaching to $w$ is not allowed, and attaching the new vertex to any other vertex  always gives different labelled trees. Also, for two different labelled trees of size $m$, any two resulting trees of size $m+1$ are different, and any tree of size $m+1$ can be constructed in exactly one way from exactly one tree of size $m$. Thus, the fraction of trees which is forbidden increases as $m$ increases. Since by~\eqref{eq:expdecrease}, the number of all trees decreases exponentially in $m$, and by forbidding certain trees the fraction of forbidden trees of a given size also increases with $m$, the sizes of connected components clearly also decrease exponentially.

Hence, denoting by $C_w^{m}$ the binary random variable which is $1$ if the connected component of $w$ has size $m \geq \ell+1$, we have for some $0 < \gamma < 1$
\begin{eqnarray*}
\E{J_k J_{\ell}}&=&\sum_{m \geq \ell+1} \E{J_k}\E{J_{\ell}} \Prob{\left(C_v^{m}=1\right)}(1+O(m/n)) \\
&= & \E{J_k}\E{J_{\ell}} \sum_{m \geq \ell+1} \gamma^{m}(1+O(m/n)) \\
&=&  \E{J_k}\E{J_{\ell}} \left( 1 + O(\ell/n) \right),
\end{eqnarray*}
where we used for the last line again the fact that $\sum_{m \geq \ell+1}\Prob(C_v^m=1)=\sum_{m \geq \ell+1}\gamma^m=1$.

For the remaining terms such as $O\left(\sum_k \E{P_k Z}\right)$ and $O\left(\sum_k \E{J_k Z}\right)$ observe the following: since $J_k$ and $P_k$ are only nonzero for trees,  by definition of $J_k$ and $P_k$, the contribution of these terms has to come from different connected components. In the case of $P_k$, $k$ vertices are forbidden, and one obtains using~\eqref{eq:expdecreaseP}, $$\sum_k \E{P_k}\E{Z}\left(1+O\left(k/n\right)\right),$$ which by Lemma~\ref{lem:notrees} can be bounded by $O(n)$.

In the case of $J_k$, by the same argument as in the case of the contribution of $\E{J_k J_{\ell}}$, the sizes of connected components decrease exponentially, given that they are at least $k+1$, and then the same result holds as well. The contribution of $\E{Z^2}$ can be bounded by all pairs of indicator variables belonging to a connected component which is not a tree. For each such a pair, the probability is at most $O\left(1/n\right)$, since this is the probability for one vertex to be not in a tree, and as there are at most $n^2$ pairs, this contribution can also be bounded by $O(n)$.

Finally, by~\eqref{eq:expdecreaseP} and by the above argument for the contribution of $\E{J_k J_{\ell}}$, both $\E{P_k}$ and $\E{J_k}$ decrease exponentially in $k$, and thus the cumulative errors of the terms
\begin{align*}
\sum_k \left(\E{I J_k}-\E{I}\E{J_k}\right), \\
\sum_k \left(\E{I P_k}-\E{I}\E{P_k}\right), \\ 
\sum_k \left(\E{L J_k}-\E{L}\E{J_k}\right), \\ 
\sum_k (\E{L P_k}-\E{L}\E{P_k}), \\
\sum_{k,\ell} \left(\E{J_k J_{\ell}}-\E{J_k}\E{J_{\ell}}\right), \\
\sum_{k,\ell} \left(\E{P_k P_{\ell}}-\E{P_k}\E{P_{\ell}}\right), \\
\sum_{k,\ell} (\E{P_k J_{\ell}}-\E{P_k}\E{J_{\ell}})
\end{align*} are still at most $O(n)$. Hence, the proof of the lemma is finished.
\end{proof}

To conclude the proof, we must show the normal limiting distribution stated in Theorem~\ref{thm:main-Gnp}. We apply Stein's Method from Theorem~\ref{thm:Stein} in the following way: set $W=\frac{\beta(G) - \E{\beta(G)}}{\sqrt{\V{\beta(G)}}}$, and write $W=\sum_{v \in V} X_v$, where $X_v=\frac{Y_v - \E{Y_v}}{\sqrt{\V{\beta(G)}}}$ with $Y_v$ being the indicator variable being $1$ if the vertex $v$ is added to a fixed minimal resolving set (chosen uniformly at random from all minimal resolving sets), and $0$ otherwise. Then $\E{X_v}=0$ and $\E{W^2}=1$. Moreover, write $W=W_i+Z_i$, for any $i \in [n]$,  where $Z_i=\sum_{k \in K_i} X_k$, with $K_i \subseteq [n]$ being the set of indices of those vertices belonging to the same connected component as the vertex with index $i$. Clearly, $W_i$ is independent of both $X_i$ and $Z_i$, and all random variables have finite variance.  For the calculation of $\varepsilon$ as in Theorem~\ref{thm:Stein}, note that for any $i,k,\ell$ belonging to the same connected component, we have
 $$
\E{|X_iX_kX_{\ell}|}=\E{\left|\frac{(Y_i-\E{Y_i})(Y_k-\E{Y_k})(Y_{\ell}-\E{Y_{\ell}})}{(\V{\beta})^{3/2}}\right|} \leq \frac{1}{(\V{\beta(G)})^{3/2}}
$$
and similarly also $$\E{|X_iX_k|}\E{|X_{\ell}|} \leq (\V{\beta(G)})^{-3/2}.$$
By~\eqref{eq:expdecrease}, the probability of belonging to a connected component of size $m$ decreases exponentially in $m$, and hence for the total contribution to $\varepsilon$ we have
$$
\varepsilon = \sum_{i} \sum_{k,\ell \in K_i}2\left(\E{|X_iX_kX_{\ell}|}+\E{|X_iX_k|}\E{|X_{\ell}|} \right) = O\left(\frac{n}{(\V{\beta(G)})^{3/2}}\right).$$
By Lemma~\ref{lem:var}, this gives $\varepsilon=O({n}^{-1/2}),$ and by Theorem~\ref{thm:Stein}, the second part of (ii) of Theorem~\ref{thm:main-Gnp} now follows.\\\\
\begin{remark}
An anonymous referee pointed out the possibility of applying the generating function techniques (in particular the system of equations~\eqref{eq:system1}) also to the $G(n,p)$ model.
The sketch of the idea is the following: denote as in Section \ref{subsec:system} by $T(x,y)$  the generating function of unrooted trees where $x,\,y$ mark vertices and the metric dimension, respectively.
Since the expected number of occurrences of each given tree of size $k$ is equal to $\frac{n}{c}\frac{(c e^{-c})^k}{k!}(1+o(1))$, we have $\E{\beta(G(n,c/n))}=\frac{n}{c} T_y(ce^{-c},1)(1+o(1))$, with $T_y$ denoting the derivative of $T$ with respect to $y$. In this way, the leading constant of the linear term in $\E{\beta(G(n,c/n))}$ can be calculated using generating function techniques.

When trying to do the same with the variance, however, there are some technical details that limit this approach: first, a more precise expression for the expected number of trees of size $k$ is needed, since for the second moment calculations all terms that are at least constant play a non-negligible role. While Stirling's formula can be applied to get a more precise estimation of the expected number of occurrences of a given tree - the term $(1+o(1))$ can be replaced by $1+O(k^2/n)$, and even as $1+c_0/n + c_1 k/n +c_2 k^2/n + o(1/n)$ for some explicit values of $c_0,c_1,c_2$ -  these constants would have to be suitably incorporated into an exponential type generating function (exponential in $k$). While this seems tedious, but still doable, second, even worse, since in expectation there is a constant number of vertices in unicyclic components, we would have to have a Slater-type characterization for the metric dimension of those.

Unfortunately, we are not aware of such a characterization, and using generation functions we could at best also get at most that the variance is linear, without finding the leading constant for the linear term. We thus opted for the classical second moment method.

\end{remark}
\paragraph{\textbf{Acknowledgements.}} This paper started while the first author visited ICMAT in Madrid in January 2013. He thanks this institution for their hospitality. The second author also thanks the support and hospitality of Charles University in Prague, where the final stage of this work was achieved. The authors would like to thank the anonymous referees whose useful comments helped to improve the presentation of the paper.

\end{document}